\documentclass[11pt]{amsart}
\usepackage{array,epsfig}

\usepackage[style=alphabetic,sorting=none, maxcitenames=10, maxbibnames=10]{biblatex}
\usepackage{IEEEtrantools}

\usepackage{caption}
\usepackage{amsmath}
\usepackage{amsfonts}
\usepackage{amssymb}
\usepackage{amsxtra}
\usepackage{amsthm}
\usepackage{mathrsfs}
\usepackage{color}
\usepackage[dvipsnames]{xcolor}
\usepackage{tikz}
\usepackage{tikz-cd}
\usepackage{mathabx,epsfig}
\usepackage[matrix,arrow]{xy}
\usepackage{comment}
\usepackage{enumerate}
\usepackage{mathtools}
\usepackage{ytableau}
\usepackage{hyperref}
\usepackage{stmaryrd}
\usepackage{extarrows}

\hypersetup{
    colorlinks=true,
    linkcolor=blue,
    filecolor=magenta,      
    urlcolor=black,
    citecolor=teal
}
\theoremstyle{plain}

\newtheorem{theorem}{Theorem}[section]

\newtheorem{lemma}[theorem]{Lemma}
\newtheorem{prop}[theorem]{Proposition}

\theoremstyle{definition}

\newtheorem{remark}[theorem]{Remark}

\DeclareMathOperator{\inv}{inv}
\newcommand{\lin}[1]{\langle #1 \rangle}
\newcommand{\lra}{\longrightarrow}

\newcommand{\bb}[1]{\mathbb{#1}}
\newcommand{\mc}[1]{\mathcal{#1}}
\newcommand{\wtl}[1]{\widetilde{#1}}

\newcommand{\Z}{\bb{Z}}
\newcommand{\Q}{\bb{Q}}

\newcommand{\T}{\bb{T}}
\newcommand{\cT}{\mathcal{T}}

\newcommand{\m}{\mathbf{m}}

\newcommand{\bl}{\bullet}

\newcommand{\bd}{\partial}

\newcommand{\sgn}{\mathrm{sgn}}
\newcommand{\rk}{\mathrm{rk}}

\newcommand{\id}{\mathrm{id}}

\newcommand{\Set}{\mathrm{Set}}

\newcommand{\St}{\mathrm{St}}
\newcommand{\Sh}{\mathrm{Sh}}

\newcommand{\ST}{\mathrm{ST}}

\usepackage[OT2,T1]{fontenc}
\DeclareSymbolFont{cyrletters}{OT2}{wncyr}{m}{n}
\DeclareMathSymbol{\Sha}{\mathalpha}{cyrletters}{"58}

\newcommand{\Tot}{\mathrm{Tot}}

\newcommand{\op}{\mathrm{op}}
\newcommand{\es}{\mathrm{es}}

\newcommand{\llb}{\llbracket}
\newcommand{\rrb}{\rrbracket}

\newcommand{\Stab}{\mathrm{Stab}}
\newcommand{\Ind}{\mathrm{Ind}}

\newcommand{\GL}{\mathrm{GL}}

\newcommand{\EQ}{^Q \! E}

\title[Hopf algebra structures on Steinberg homology]{The Quillen and sharbly Hopf algebra structures on Steinberg homology coincide}
\author{Urshita Pal}
\author{Sam Payne}
\date{\today}

\begin{document}

\begin{abstract}
Two recent constructions independently give bigraded commutative Hopf algebra structures on the homology of $\GL(\Z)$ with rational Steinberg coefficients. One, by Ash--Miller--Patzt, uses the sharbly resolution. The other, by Brown--Chan--Galatius--Payne, uses the Quillen spectral sequence and rank-filtered maps on $BK(\Z)$. We prove that the two structures coincide. Both are induced by join and parabolic restriction operations on Steinberg modules.
\end{abstract}

\maketitle

\setcounter{tocdepth}{1}
\tableofcontents

\section{Introduction}\label{sec:intro}

Let $\St_n$ denote the rational Steinberg module of $\GL_n(\Q)$, i.e., the reduced rational homology of the Tits building of $\Q^n$. By convention $\St_0 = \Q$. Consider the bigraded vector space
\[ \mc{H} \;=\; \bigoplus_{n,k\ge 0} H_k\bigl(\GL_n(\Z);\, \St_n\bigr), \]
in which $H_k(\GL_n(\Z); \St_n)$ has bidegree $(n,k)$. The recent literature contains two  constructions of bigraded commutative Hopf algebra structures on $\mc{H}$. Ash--Miller--Patzt \cite{ash2024hopf} define a product and a coproduct by explicit formulas on the sharbly resolution of the Steinberg modules. We call the resulting structure the \emph{sharbly Hopf structure}. Brown--Chan--Galatius--Payne \cite{brown2024hopf} construct rank-filtered product and coproduct maps on the Waldhausen model of $BK(\Z)$, inducing a bigraded Hopf algebra structure on each page of the Quillen spectral sequence $\EQ^{*}_{*, *}$. We identify $\EQ^1_{*,*} \cong \mc{H}$ by an isomorphism essentially due to Quillen; see \cite[Proposition 2.15]{brown2024hopf} and Section~\ref{sec:rigidified} for the precise identification. We call the Hopf algebra structure on $\mc{H}$ induced by this identification the \emph{Quillen Hopf structure}.

\begin{theorem}\label{thm:main}
The sharbly Hopf structure and the Quillen Hopf structure on $\mc{H}$ coincide. 
\end{theorem}

\noindent In other words, their respective products, coproducts, units, counits, and antipodes are equal as maps of bigraded vector spaces. To prove this, we show that both structures are induced by a single pair of equivariant maps on Steinberg modules. The first is the join product
\[ \mu\colon\; \St_m \otimes \St_n \lra \St_{m+n}, \]
which sends a tensor product of apartment classes to the apartment class of the concatenated basis (Proposition~\ref{prop:product}). The second is the parabolic restriction coproduct
\begin{equation}\label{eq:coprod-sum}
\delta\colon\; \St_n \lra \bigoplus_{U \subseteq \Q^n} \St(U) \otimes \St(\Q^n/U),
\end{equation}
given on apartment classes by
\begin{equation} \label{eq:coproduct}
    \delta\llb v_1,\dots,v_n \rrb \;=\; \sum_{S \subseteq \{ 1, \ldots, n \}} \sgn(\sigma_S)\; \llb v_{i_1},\dots,v_{i_p} \rrb \otimes \llb \bar v_{j_1},\dots,\bar v_{j_{n-p}} \rrb. 
\end{equation}     
Here, the elements of $S$ and its complement are ordered so that $i_1 < \cdots < i_p$ and $j_1 < \cdots < j_{n-p}$, and $\sigma_S := (i_1, \ldots, i_p, j_1, \ldots, j_{n-p})$ is the corresponding shuffle permutation; the term of \eqref{eq:coproduct} indexed by $S$ lives in the summand of \eqref{eq:coprod-sum} for $U = \langle v_{i_1}, \ldots, v_{i_p} \rangle$. This formula serves as a bridge between the Quillen and sharbly coproducts. The sharbly coproduct of Ash--Miller--Patzt applied to a sharbly of minimal size is given by essentially the same expression, while in the Quillen coproduct constructed by Brown--Chan--Galatius--Payne, it emerges from edgewise subdivision of $BK(\Z)$ and passage to the associated graded of the rank filtration (Proposition~\ref{prop:coproduct}).

The maps $\mu$ and $\delta$ induce a Steinberg product and coproduct on $\mc{H}$, as explained in Section~\ref{sec:reference}.  In Section~\ref{sec:amp}, we show that the sharbly product and coproduct are equivariant lifts of $\mu$ and $\delta$ to the sharbly resolution (Theorem~\ref{thm:amp}), and conclude that they agree with the Steinberg product and coproduct. In Section~\ref{sec:bcgp}, we show that the product and coproduct of the Quillen Hopf structure are equivariant lifts of $\mu$ and $\delta$ to a bounded-below complex of free modules with homology $\St_n$ concentrated in degree $n$, obtained from a rigidification of the Waldhausen construction. Via a homological comparison result we deduce that they also coincide with the Steinberg product and coproduct.  Care is required to keep track of signs; we follow the Koszul sign convention throughout. 


\medskip

\noindent{\textbf{Acknowledgements.}} UP thanks Jeremy Miller and Peter Patzt for helpful discussions on the parabolic restriction coproduct. SP is supported in part by NSF DMS--2542134 and a Simons Fellowship.

\section{Preliminaries}\label{sec:prelim}

In this section, we recall some basic notions that will be used throughout. We start with Steinberg modules and their apartment classes, the simplicial model of the doubly suspended Tits building, and a few simplicial tools: edgewise subdivision, and the Eilenberg--Zilber and Alexander--Whitney maps. The section concludes with a comparison lemma from homological algebra.

\subsection{Buildings and Steinberg modules} \label{sec:tools}

Let $V$ be a finite-dimensional $\Q$-vector space. The \emph{Tits building} $\T(V)$ is the geometric realization of the poset of nonzero proper subspaces of $V$, ordered by inclusion. It carries an action of $\GL(V)$, and is homotopy equivalent to a wedge of spheres of dimension $n-2$, where $n := \dim_\Q V$. We refer the reader to \cite[Chapter~IV]{brown1989buildings} for further details.

The \emph{rational Steinberg module} is $$\St(V) := \wtl{H}_{n-2}(\T(V); \Q).$$ The \emph{apartment} associated to an ordered basis $v_1, \dots, v_n$ is the $(n-2)$-cycle 
\[ \sum_{\sigma \in \Sigma_n} \sgn(\sigma)\,\bigl\{\lin{v_{\sigma(1)}} \subset \lin{v_{\sigma(1)}, v_{\sigma(2)}} \subset \cdots \subset \lin{v_{\sigma(1)}, \dots, v_{\sigma(n-1)}}\bigr\}, \]
and $\St(V)$ is generated by the classes of such apartments. Note that $\T(\Q)$ is empty, so $\St(\Q) = \wtl{H}_{-1}(\emptyset) \cong \Q$.  Also, by convention, $\St(0) := \Q$.

Many of our computations with Steinberg modules will be carried out in the following simplicial model of the double suspension of the building. Consider the simplicial set whose $k$-simplices are the flags $0 \subseteq V_0 \subseteq V_1 \subseteq \cdots \subseteq V_k \subseteq V$ of subspaces of $V$, with face maps $d_i$ deleting $V_i$ and degeneracies $s_i$ repeating $V_i$, for $0 \leq i \leq k$. Let $\Sigma\ST(V)$ be the pointed simplicial set obtained by collapsing the simplicial subset of flags with $V_0 \neq 0$ or $V_k \neq V$ to a distinguished basepoint. (When $V = 0$ this subset is empty, and the quotient is understood to adjoin a disjoint basepoint, so that $\Sigma\ST(0) \cong S^0$). In comparison with $\T(V)$, allowing $V_0 = 0$ cones off the building, and allowing $V_k = V$ cones off the result; collapsing the flags where either one of the equalities doesn't hold realizes $\Sigma\ST(V)$ as the double unreduced suspension of $\T(V)$. The reduced homology of $\Sigma\ST(V)$ is concentrated in degree $n$, and there is a $\GL(V)$-equivariant isomorphism $\wtl{H}_n(\Sigma\ST(V)) \cong \St(V)$. We write
\begin{equation} \label{eq:apartment}
\llb v_1, \dots, v_n \rrb \;:=\; \sum_{\sigma \in \Sigma_n} \sgn(\sigma)\, x_\sigma, \qquad x_\sigma := \bigl(0 \subset \lin{v_{\sigma(1)}} \subset \cdots \subset \lin{v_{\sigma(1)}, \dots, v_{\sigma(n-1)}} \subset V\bigr),
\end{equation}
for the apartment class in $\wtl{H}_n(\Sigma\ST(V))$ associated to an ordered basis $v_1, \dots, v_n$; here $x_\sigma$ is the nondegenerate $n$-simplex of $\Sigma\ST(V)$ on the indicated flag.

\subsection{Edgewise subdivision.} Let $\mathbf{\Delta}$ denote the category of finite nonempty totally ordered sets and order-preserving maps. When no confusion is possible, we tacitly replace $\mathbf{\Delta}$ with the skeletal subcategory with objects $[p] := \{0 < 1 < \cdots < p\}$, for nonnegative integers $p$.

A simplicial set is a functor $X\colon \mathbf{\Delta}^{\op} \to \Set$, and we write $X_p := X([p])$. Its \emph{edgewise subdivision}, denoted $\es(X)$, is given by composition with the functor $\mathbf{\Delta}^{\op} \to \mathbf{\Delta}^{\op}$ induced by $$[p] \mapsto [p] \sqcup [p] \cong [2p+1],$$ the ordered concatenation. Thus, $\es(X)_p = X_{2p+1}$. 

Following \cite[Section~3]{brown2024hopf}, to which we refer the reader for details and further references, we recall that there is a natural homeomorphism of geometric realizations $|\es(X)| \cong |X|$, affine on simplices, sending each vertex of $\es(X)_0 = X_1$ to the midpoint of the corresponding edge. 

For a bisimplicial set $Y$, i.e., a functor $Y \colon \mathbf{\Delta}^{\op} \times \mathbf{\Delta}^{\op} \to \Set$, we write $\es(Y)$ for edgewise subdivision in the first simplicial direction, i.e., $\es(Y)_{p,q} := Y_{2p+1,q}$. Again the geometric realization $|\es(Y)|$ is naturally homeomorphic to $|Y|$.

\subsection{Chains, Eilenberg--Zilber, and Alexander--Whitney maps.}
\label{sec:chains}

For a simplicial set $X$, let $C_n(X)$ be the normalized rational chain group, i.e., the quotient of the $\Q$-vector space on $X_n$ by the subspace spanned by degenerate simplices. Its differential is $\partial=\sum_i (-1)^i d_i$. If $X$ is pointed, let $\wtl{C}_*(X) := C_*(X)/C_*(*)$. Thus $\wtl{C}_*(X)$ is obtained by setting the class of the basepoint to zero. For a bisimplicial set, we take normalized chains in both simplicial directions. If $A_{p,q}$ is the resulting first-quadrant double complex, with horizontal and vertical differentials $d_h$ and $d_v$, we use the totalization convention
\[
\Tot_r(A)=\bigoplus_{p+q=r}A_{p,q},\qquad d_{\Tot}=d_h+(-1)^p d_v \quad\text{on }A_{p,q}.
\]
When Eilenberg--Zilber or Alexander--Whitney is applied in both simplicial directions, we use the Koszul interchange
\[
(A\otimes B)\otimes(C\otimes D)\longrightarrow(A\otimes C)\otimes(B\otimes D),\qquad (a\otimes b)\otimes(c\otimes d)\longmapsto(-1)^{|b||c|}(a\otimes c)\otimes(b\otimes d).
\]
Thus, on elements of bidegrees $(p,q)$ and $(r,s)$, the interchange contributes the sign $(-1)^{qr}$.

For $p,q \geq 0$, let $\mathfrak{S}(p,q)$ be the set of $(p,q)$-shuffles, i.e., the set of pairs
\[
(a,b)=(a_1<\cdots<a_p,\; b_1<\cdots<b_q)
\]
such that $\{ 1, \ldots, p+q \} = \{ a_1, \ldots, a_p\} \sqcup \{ b_1, \ldots, b_q \}$. 

We write $\rho_{a,b} \in \Sigma_{p+q}$ for the permutation given by $a_j \mapsto j$ and $b_j \mapsto p + j$. Note that $\rho_{a,b}$  is the inverse of the shuffle permutation $(a_1, \ldots, a_p, b_1, \ldots, b_q)$; in particular, $\rho_{a,b}$ and the shuffle permutation given by $(a,b)$ have the same sign, which we denote $\sgn(a,b)$.  Composition of permutations $(\sigma,\tau,(a,b)) \mapsto (\sigma\oplus\tau)\circ\rho_{a,b}$ gives a bijection from $\Sigma_p\times\Sigma_q\times\mathfrak{S}(p,q)$ to $\Sigma_{p+q}$, and  
\[
\sgn(\sigma)\cdot\sgn(\tau)\cdot\sgn(a,b) \;=\; \sgn\bigl((\sigma\oplus\tau)\circ\rho_{a,b}\bigr).
\]
Recall that $s_i$ denotes the $i$th degeneracy map on a simplicial set.  Consider the iterated degeneracies
\[
s_b=s_{b_q-1}\cdots s_{b_1-1},
\quad \mbox{and} \quad
s_a=s_{a_p-1}\cdots s_{a_1-1}.
\]
The Eilenberg--Zilber map $\mathrm{EZ} \colon C_*(X) \otimes C_*(Y) \to C_*(X \times Y)$ is given by
\[
\mathrm{EZ}(x\otimes y)
 =
 \sum_{(a,b)\in \operatorname{Sh}(p,q)} \sgn(a,b)\ (s_bx,s_ay), \quad \mbox{ for } \quad (x,y) \in X_p \times Y_q.
\]

For an $n$-simplex $z$ and $p+q=n$, set
\[
\operatorname{fr}_p(z):=d_{p+1}\cdots d_n z,
\quad \mbox{ and } \quad
\operatorname{bk}_q(z):=d_0^{\,p}z .
\]
These are the front $p$-face and the back $q$-face of $z$. The Alexander-Whitney map $\mathrm{AW}$ from \mbox{$C_*(X \times Y)$} to $C_*(X) \otimes C_*(Y)$ is given by
\[
\mathrm{AW}(x,y) = \sum_{p+q=n} \operatorname{fr}_p(x)\otimes \operatorname{bk}_q(y), \quad \mbox{ for } \quad (x,y)\in X_n\times Y_n .
\]

The same formulas induce maps on reduced chains of pointed simplicial sets, similarly denoted
\[
\mathrm{EZ}\colon \wtl{C}_*(X)\otimes \wtl{C}_*(Y) \longrightarrow \wtl{C}_*(X\wedge Y),
\quad \mbox{ and } \quad 
\mathrm{AW}\colon \wtl{C}_*(X\wedge Y) \longrightarrow \wtl{C}_*(X)\otimes \wtl{C}_*(Y).
\]
The induced maps on homology are mutually inverse \cite[\S 8.5]{weibel1994introduction}. 

\subsection{Two lemmas in group homology}\label{sec:comparison}

Let $G$ be a group and let $M$ be a $\Q[G]$-module. We regard group homology as the left derived functor of coinvariants,
\[
H_*(G;M)=\operatorname{Tor}^{\Q[G]}_*(\Q,M)=H_*\bigl(\Q\otimes^{\mathbf L}_{\Q[G]}M\bigr),
\]
computed by tensoring any flat resolution of $M$ with $\Q$ over $\Q[G]$ \cite[Lemma~3.2.8]{weibel1994introduction}. More generally, for a bounded-below complex $T_\bl$ of $\Q[G]$-modules, set $H_*(G;T_\bl):=H_*\bigl(\Q\otimes^{\mathbf L}_{\Q[G]}T_\bl\bigr)$. If the terms of $T_\bl$ are flat, then the coinvariant complex $(T_\bl)_G:=\Q\otimes_{\Q[G]}T_\bl$ computes $H_*(G;T_\bl)$.

Let $\alpha\colon G'\to G$ be a group homomorphism and let $f\colon M'\to M$ be an $\alpha$-equivariant map. Write $\alpha^*M$ for $M$ with its induced $\Q[G']$-action. The natural map $(\alpha^*M)_{G'}\to M_G$, given by $1\otimes m\mapsto 1\otimes m$, derives to a morphism which, together with $f$, gives the induced map on homology, denoted $(\alpha,f)_*\colon H_*(G';M')\to H_*(G;M)$. For a module $M$ and an integer $i$, let $M[-i]$ denote the complex given by $M$ in degree $i$.  Let $D(R)$ denote the derived category of complexes of $R$-modules. 

\medskip


\begin{lemma}\label{lem:comparison}
Assume $T'_\bl$ and $T_\bl$ are bounded below complexes of flat modules, both with homology concentrated in degree $N$. Let $M':=H_N(T'_\bl)$ and $M:=H_N(T_\bl)$. There are natural isomorphisms
\[
(T'_\bl)_{G'}\simeq\bigl(\Q\otimes^{\mathbf L}_{\Q[G']}M'\bigr)[-N],\qquad (T_\bl)_G\simeq\bigl(\Q\otimes^{\mathbf L}_{\Q[G]}M\bigr)[-N]
\]
in $D(\Q)$. If $\tau\colon T'_\bl\to T_\bl$ is $\alpha$-equivariant and induces $f\colon M'\to M$ on $H_N$, then the map on coinvariant complexes induced by $\tau$ is the derived-coinvariant map associated to $(\alpha,f)$. Consequently,
\[
H_{N+k}\bigl((T'_\bl)_{G'}\bigr)\cong H_k(G';M'),\qquad H_{N+k}\bigl((T_\bl)_G\bigr)\cong H_k(G;M),
\]
and $(\alpha,f)_*$ is the map induced by $\tau$ on homology.
\end{lemma}

\begin{proof}
Canonical truncation gives functorial isomorphisms $T'_\bl\simeq M'[-N]$ in $D(\Q[G'])$ and $T_\bl\simeq M[-N]$ in $D(\Q[G])$. Ordinary tensor product with a bounded-below complex of flat modules computes the derived tensor product, by \cite[Lemma~10.6.2 and Theorem~10.6.3]{weibel1994introduction}. Hence
\[
(T_\bl)_G=\Q\otimes_{\Q[G]}T_\bl\simeq\Q\otimes^{\mathbf L}_{\Q[G]}T_\bl\simeq\Q\otimes^{\mathbf L}_{\Q[G]}M[-N]=\bigl(\Q\otimes^{\mathbf L}_{\Q[G]}M\bigr)[-N],
\]
and similarly over $G'$. To verify the assertion about $\tau$, use the standard functorial bar resolutions. The homomorphism $\alpha$ and the chain map $\tau$ give a map between the corresponding bar double complexes. These complexes compute the two derived tensor products, and their augmentations to $(T'_\bl)_{G'}$ and $(T_\bl)_G$ are quasi-isomorphisms because the terms of $T'_\bl$ and $T_\bl$ are flat. Since $H_N(\tau)=f$, canonical truncation identifies the map of bar complexes with the derived-coinvariant map associated to $(\alpha,f)$. The asserted naturality and the formulas on homology follow.
\end{proof}

The next lemma is a related statement for chain-level lifts, when the terms of $T'_\bl$ are projective. Let $K(R)$ denote the homotopy category of complexes of $R$-modules.

\begin{lemma}\label{lem:chain-lifting}
Assume $T'_\bl$ is a bounded-below complex of projective $\Q[G']$-modules and $T_\bl$ is a bounded below complex of $\Q[G]$-modules, both with homology concentrated in degree $N$. Then
\[
H_N\colon\operatorname{Hom}_{K(\Q[G'])}(T'_\bl,\alpha^*T_\bl)\xrightarrow{\;\cong\;}\operatorname{Hom}_{\Q[G']}\bigl(H_N(T'_\bl),\alpha^*H_N(T_\bl)\bigr).
\]
Thus every $\alpha$-equivariant map $H_N(T'_\bl) \to H_N(T_\bl)$ has an $\alpha$-equivariant chain-level lift, and the lift is unique up to $\alpha$-equivariant chain homotopy.
\end{lemma}

\begin{proof}
Morphisms from a bounded-below complex of projective modules in the homotopy and derived categories agree, by  \cite[Corollary~10.4.7]{weibel1994introduction}. Canonical truncation gives $T'_\bl\cong H_N(T'_\bl)[-N]$ and $\alpha^*T_\bl\cong\alpha^*H_N(T_\bl)[-N]$ in $D(\Q[G'])$. Therefore
\begin{align*}
\operatorname{Hom}_{K(\Q[G'])}(T'_\bl,\alpha^*T_\bl)&\cong\operatorname{Hom}_{D(\Q[G'])}(T'_\bl,\alpha^*T_\bl),\\
&\cong\operatorname{Hom}_{D(\Q[G'])}\bigl(H_N(T'_\bl)[-N],\alpha^*H_N(T_\bl)[-N]\bigr),
\end{align*}
which is $\operatorname{Hom}_{\Q[G']}\bigl(H_N(T'_\bl),\alpha^*H_N(T_\bl)\bigr)$, and the result follows.
\end{proof}

\begin{remark}
For $N=0$, Lemma~\ref{lem:comparison} is the usual comparison result for maps of flat resolutions, which is used for the sharbly resolution in Section~\ref{sec:amp}. Its shifted form is used for the rigidified Waldhausen complex in Section~\ref{sec:rigidified}. Lemma~\ref{lem:chain-lifting} supplies the chain-level lifts and homotopies used in the proof of Lemma~\ref{lem:lift}.
\end{remark}

\section{The Steinberg product and coproduct}\label{sec:reference}

This section constructs the product and coproduct to which those in the Quillen and sharbly Hopf structures will be compared. We define the join product $\mu$ and the parabolic restriction coproduct $\delta$ on the Steinberg modules, compute each on apartment classes (Propositions \ref{prop:product} and \ref{prop:coproduct}), and describe the induced product and coproduct on $\mc{H}$ (Section~\ref{sec:induced}). Only the apartment coproduct formula (Proposition~\ref{prop:coproduct}) requires a proof of any length.

\subsection{A product on Steinberg modules} \label{sec:product}

Recall from Section~\ref{sec:tools} that $\Sigma\ST(V)$ denotes the doubly suspended Tits building of a rational vector space $V$.  We write $\Sigma\ST_n := \Sigma\ST(\Q^n)$. The reduced homology of $\Sigma\ST_n$ is the Steinberg module, concentrated in degree $n$, denoted $\St_n := \St(\Q^n)$.

Let $\phi\colon \Q^m \to \Q^{m+n}$ and $\psi\colon \Q^n \to \Q^{m+n}$ be the embeddings onto the first $m$ and the last $n$ coordinates. The \emph{flag-sum map} is the map of pointed simplicial sets
\[ \theta_{m,n}\colon\; \Sigma\ST_m \times \Sigma\ST_n \lra \Sigma\ST_{m+n}, \qquad (V_\bl,\, W_\bl) \longmapsto \phi(V_\bl) + \psi(W_\bl). \]
If $V_0 \neq 0$ or $W_0 \neq 0$ then $\phi(V_0) + \psi(W_0) \neq 0$. Likewise, if $V_k \neq \Q^m$ or $W_k \neq \Q^n$ then $\phi(V_k) + \psi(W_k) \neq \Q^{m+n}$. Thus $\theta_{m,n}$ sends the collapsed subspace of either factor to $\ast$, and hence factors through the smash product $\Sigma\ST_m \wedge \Sigma\ST_n$.  Composing with $\mathrm{EZ}$, one obtains a reduced chain map
\[ 
\wtl{C}_*(\Sigma\ST_m) \otimes \wtl{C}_*(\Sigma\ST_n) \lra \wtl{C}_*(\Sigma\ST_{m+n}).
\]
The homology on each side is concentrated in degree $m+n$, and we denote the induced map
\[ \mu\colon\; \St_m \otimes \St_n \lra \St_{m+n}. \]
It is equivariant for the block diagonal embedding $\GL_m(\Q) \times \GL_n(\Q) \to \GL_{m+n}(\Q)$. 

\begin{prop}\label{prop:product}
For ordered bases $v_1, \dots, v_m$ of $\Q^m$ and $w_1, \dots, w_n$ of $\Q^n$,
\[ \mu\bigl(\llb v_1, \dots, v_m \rrb \otimes \llb w_1, \dots, w_n \rrb\bigr) \;=\; \llb \phi(v_1), \dots, \phi(v_m),\, \psi(w_1), \dots, \psi(w_n) \rrb. \]
\end{prop}

\begin{proof}
Set $u_i = \phi(v_i)$ for $1 \leq i \leq m$ and $u_{m+j} = \psi(w_j)$ for $1 \leq j \leq n$, a basis of $\Q^{m+n}$. Write $y_\sigma$, $\sigma \in \Sigma_m$, for the simplices of $\Sigma\ST_m$ underlying $\llb v_1, \dots, v_m \rrb$ as in \eqref{eq:apartment}, and $z_\tau$, $\tau \in \Sigma_n$, for those of $\Sigma\ST_n$ underlying $\llb w_1, \dots, w_n \rrb$. By the Eilenberg--Zilber formula of Section~\ref{sec:chains},
\[ (\theta_{m,n})_*\,\mathrm{EZ}\bigl(y_\sigma \otimes z_\tau\bigr) \;=\; \sum_{(a,b)} \sgn(a,b)\;\theta_{m,n}\bigl(s_b\, y_\sigma,\; s_a\, z_\tau\bigr), \]
the sum over $(m,n)$-shuffles. In the $(a,b)$-term, the flag $s_b\, y_\sigma$ grows by $\phi(v_{\sigma(j)})$ at step $a_j$ and is otherwise constant, while $s_a\, z_\tau$ grows by $\psi(w_{\tau(j)})$ at step $b_j$; their levelwise sum is the flag of initial spans of $u_{\rho(1)}, \dots, u_{\rho(m+n)}$, where $\rho = (\sigma\oplus\tau)\circ\rho_{a,b}$ in the notation of Section~\ref{sec:chains}. This flag is proper and nondegenerate, its spans strictly increasing because the $u_i$ form a basis, so it is the simplex $x_\rho$ of $\Sigma\ST_{m+n}$. By the bijection of Section~\ref{sec:chains}, summing $\sgn(\sigma)\sgn(\tau)\sgn(a,b) = \sgn(\rho)$ over all triples gives
\[ (\theta_{m,n})_*\,\mathrm{EZ}\bigl(\llb v_1,\dots,v_m \rrb \otimes \llb w_1,\dots,w_n \rrb\bigr) \;=\; \sum_{\rho \in \Sigma_{m+n}} \sgn(\rho)\, x_\rho \;=\; \llb u_1, \dots, u_{m+n} \rrb. \qedhere \]
\end{proof}


\subsection{A coproduct on Steinberg modules}\label{sec:coproduct}

Recall from Section~\ref{sec:tools} that $\Sigma\ST(V)$ is pointed, with basepoint the collapsed simplicial subset of flags having $V_0 \neq 0$ or $V_k \neq V$. The basepoint of its edgewise subdivision is likewise the collapsed simplicial subset of flags $(V_0 \subseteq \cdots \subseteq V_{2p+1})$ with $V_0 \neq 0$ or $V_{2p+1} \neq V$. 

Fix $n \geq 1$ and a subspace $U \subseteq \Q^n$. Define a function $r_U$ from the $p$-simplices of $\es(\Sigma\ST_n)$ to those of $\Sigma\ST(U) \wedge \Sigma\ST(\Q^n/U)$ by sending  $x = (0 \subseteq V_0 \subseteq \cdots \subseteq V_{2p+1} \subseteq \Q^n)$ to 
\begin{equation} \label{eq:rU}
r_U(x) := \bigl(V_0 \subseteq \cdots \subseteq V_p\bigr) \wedge \bigl(V_{p+1}/U \subseteq \cdots \subseteq V_{2p+1}/U\bigr) \quad\text{if } V_p \subseteq U \subseteq V_{p+1},
\end{equation}
and $r_U(x) = \ast$ otherwise.

\begin{lemma}\label{lem:splitting}
The map $r_U$ is a map of pointed simplicial sets, and $r_U(x) \neq \ast$ only if $V_p = U = V_{p+1}$.
\end{lemma}

\begin{proof}
Suppose $r_U(x) \neq \ast$. Then the front flag $(V_0 \subseteq \cdots \subseteq V_p)$ is not collapsed in $\Sigma\ST(U)$, so $V_p = U$, and the back flag $(V_{p+1}/U \subseteq \cdots \subseteq V_{2p+1}/U)$ is not collapsed in $\Sigma\ST(\Q^n/U)$, so $V_{p+1} = U$. This proves the second claim and shows that if $r_U$ is simplicial then it is pointed. 

We now show $r_U$ is simplicial. Every degeneracy $s_i$ leaves the middle pair $(V_p,V_{p+1})$ unchanged and therefore commutes with $r_U$ factorwise. There are no faces to check in degree $p=0$, so assume $p>0$. Every face $d_i$ with $1 \leq i \leq p-1$ also leaves the middle pair untouched. The faces $d_0$ and $d_p$ delete one of $V_p,V_{p+1}$; hence, the condition in \eqref{eq:rU} says that $r_U(d_0x)$ and $r_U(d_px)$ are not necessarily sent to $\ast$ when $V_p \subseteq U \subseteq V_{p+2}$ and when $V_{p-1} \subseteq U \subseteq V_{p+1}$, respectively. Instead, on simplices meeting these conditions, both $r_U \circ d_i$ and $d_i \circ r_U$ are given by the same factorwise formula. It remains to check that $d_0 x$ and $d_p x$ map to the basepoint whenever $x$ does. If $r_U(d_0 x) \neq \ast$, then $V_1 = 0$ and the second claim applied to $d_0 x$ says that $V_p = U = V_{p+2}$. It follows that $V_0 = 0$ and $V_{p+1} = U$, so $r_U(x) \neq \ast$.  The case of $d_p x$ is similar, and the lemma follows.
\end{proof}

By Lemma~\ref{lem:splitting}, if $r_U(x) \neq \ast$ then $U = V_p$, so at most one of the maps $r_U$ carries a given simplex away from the basepoint. They therefore assemble to a map of pointed simplicial sets
\[ r\colon\; \es(\Sigma\ST_n) \lra \bigvee_{U \subseteq \Q^n} \big( \Sigma\ST(U) \wedge \Sigma\ST(\Q^n/U)\big). \]
It is equivariant for $\GL_n(\Q)$ acting on the target by permuting the wedge summands.

Let $\varphi_0\colon \es(X) \to X$ be the natural map induced by the inclusion of $[p]$ onto the first $p+1$ elements of $[p] \sqcup [p]$. By \cite[Proposition~3.8]{brown2024hopf}, the composition of $|\varphi_0|$ with the inverse of the homeomorphism $|\es(X)| \cong |X|$ is homotopic to the identity and hence induces an isomorphism on homology. We  
define $\delta$ as the composite
\[ 
\delta\colon\; \St_n \;\xrightarrow{\;r_* \circ (\varphi_0)_*^{-1}\;}\; \bigoplus_{U \subseteq \Q^n} \wtl{H}_n\bigl(\Sigma\ST(U) \wedge \Sigma\ST(\Q^n/U)\bigr) \;\xrightarrow{\;\mathrm{AW}_*\;}\; \bigoplus_{U \subseteq \Q^n} \St(U) \otimes \St(\Q^n/U). 
\]
Note that  $\delta$ is equivariant for $\GL_n(\Q)$ acting on the target by permuting the summands. 

\begin{prop}\label{prop:coproduct}
For an ordered basis $v_1, \dots, v_n$ of $\Q^n$,
\[ \delta\llb v_1, \dots, v_n \rrb \;=\; \sum_{S \subseteq \{1, \dots, n\}} \sgn(\sigma_S)\; \llb v_{i_1}, \dots, v_{i_p} \rrb \otimes \llb \bar v_{j_1}, \dots, \bar v_{j_{n-p}} \rrb. \]
\end{prop}

\noindent Here, the elements of $S = \{i_1, \dots, i_p\}$ and its complement are ordered so that $i_1 < \cdots < i_p$ and $j_1 < \cdots < j_{n-p}$, and $\sigma_S := (i_1, \ldots, i_p, j_1, \ldots, j_{n-p})$ is the corresponding shuffle permutation; the term of \eqref{eq:coproduct} indexed by $S$ lives in the summand indexed by $U = \langle v_{i_1}, \ldots, v_{i_p} \rangle$, and $\bar v$ denotes the image of $v$ in $\Q^n/U$. Note that the improper subsets $S = \varnothing$ and $S = \{1, \dots, n\}$ contribute $1 \otimes \llb v_1, \dots, v_n \rrb$ and $\llb v_1, \dots, v_n \rrb \otimes 1$ in the summands $U = 0$ and $U = \Q^n$, under the convention $\St(0) = \Q$.

\begin{proof}
Let $A \subseteq \Sigma\ST_n$ be the simplicial subset generated by the simplices $x_\rho$ of \eqref{eq:apartment}, $\rho \in \Sigma_n$. Its non-basepoint simplices are the flags of spans of nested subsets of $\{v_1, \dots, v_n\}$. We compute $\delta\llb v_1, \dots, v_n \rrb$ by exhibiting a cycle representing $(\varphi_0)_*^{-1}\llb v_1, \dots, v_n \rrb$ in $\wtl{C}_n(\es(A))$ and then pushing it forward under $\mathrm{AW} \circ r$.

We first parametrize the top-dimensional nondegenerate simplices of $\es(A)$. A nondegenerate $q$-simplex is a flag $(0 = V_0 \subseteq \cdots \subseteq V_{2q+1} = \Q^n)$ in which, among the paired inclusions $V_{i-1} \subseteq V_i$ and $V_{q+i} \subseteq V_{q+i+1}$, at least one of each pair is strict. The top-dimensional nondegenerate simplices, for $q = n$, thus have exactly one strict inclusion per pair and middle equality $V_n = V_{n+1}$, so they are indexed as
\[ y_{\rho, \varepsilon}, \qquad \rho \in \Sigma_n, \quad \varepsilon \in \{0,1\}^n, \]
where $\varepsilon_i$ records which inclusion of the $i$th pair is strict (so that $\varepsilon_i = 0$ means that $V_{i-1} \subset V_i$ is strict), and $\rho$ is the order in which the basis vectors enter, reading upward.

We now describe a cycle representing the apartment class $\llb v_1, \ldots, v_n \rrb$ in  $\wtl{C}_n\bigl(\es(A)\bigr)$. For $\varepsilon \in \{0,1\}^n$, let $\inv(\varepsilon) = \#\{i < j : \varepsilon_i = 1, \varepsilon_j = 0\}$, and set
\[ \wtl{z} \;=\; \sum_{\rho, \varepsilon} \sgn(\rho)\,(-1)^{\inv(\varepsilon)}\; y_{\rho, \varepsilon} \;\in\; \wtl{C}_n\bigl(\es(A)\bigr). \]
We check $\bd \wtl{z} = (\sum_i (-1)^i d_i)\wtl{z} = 0$, where $d_i$ deletes $V_i$ and $V_{n+1+i}$. The interior faces $d_i$, $1 \leq i \leq n-1$, cancel in pairs via a sign-reversing transposition in $\rho$ when $\varepsilon_i = \varepsilon_{i+1}$, and via the swap changing $\inv(\varepsilon)$ by one when $\{\varepsilon_i, \varepsilon_{i+1}\} = \{0,1\}$. The boundary faces $d_0$ and $d_n$, nonzero only when $\varepsilon_1 = 1$ and $\varepsilon_n = 0$ respectively, instead pair across the chain: the strict inclusion deleted at one end reappears as the middle one, so $d_0 y_{\rho,(1,\eta)} = d_n y_{\rho,(\eta,0)}$ for $\eta \in \{0,1\}^{n-1}$. These faces enter $\bd\wtl z$ with signs $(-1)^0$ and $(-1)^n$, so the two contributions cancel precisely when $\inv(1\eta)-\inv(\eta0)-n$ is odd; and $\inv(1\eta) - \inv(\eta)$ and $\inv(\eta 0) - \inv(\eta)$ are the numbers of zeros and ones of $\eta$, summing to $n-1$, whence $\inv(1\eta)-\inv(\eta0)\equiv n-1$ and the exponent is odd. Hence $\bd \wtl{z} = 0$. 
Moreover $\varphi_0$ sends $y_{\rho,\varepsilon}$ to its front half $(V_0 \subseteq \cdots \subseteq V_n)$, nondegenerate only for $\varepsilon = 0^n$, where it is $x_\rho$; as $\inv(0^n) = 0$,
\[ \varphi_0(\wtl{z}) \;=\; \sum_\rho \sgn(\rho)\, x_\rho \;=\; \llb v_1, \dots, v_n \rrb, \]
so $[\wtl{z}\,] = (\varphi_0)_*^{-1}\llb v_1, \dots, v_n \rrb$.

Finally, we compute the chain-level push-forward of $\wtl{z}$ under $\mathrm{AW} \circ r$. Each $y_{\rho,\varepsilon}$ has $V_n = V_{n+1}$, so $r(y_{\rho,\varepsilon})$ lies in the summand $U = V_n$, and the $p$th Alexander--Whitney term of Section~\ref{sec:chains} is
\[ \bigl(V_0 \subseteq \cdots \subseteq V_p\bigr) \otimes \bigl(V_{n+1+p}/U \subseteq \cdots \subseteq V_{2n+1}/U\bigr). \]
The first factor survives in $\Sigma\ST(U)$ if and only if its inclusions are strict with $V_p = U$, and the second survives in $\Sigma\ST(\Q^n/U)$ if and only if its inclusions are strict with $V_{n+1+p} = U$. Together these say $\varepsilon = 0^p 1^{n-p}$, leaving a single term. Fix $S=\{i_1<\cdots<i_p\}$, with complement $\{j_1<\cdots<j_{n-p}\}$, and write $\rho=\sigma_S \circ (\sigma'\oplus\tau')$, where $\sigma'\in\Sigma_p$ and $\tau'\in\Sigma_{n-p}$. The surviving term is $x_{\sigma'}\otimes x_{\tau'}$ with coefficient $\sgn(\rho)$, since $\inv(0^p1^{n-p})=0$, and
\[
\sgn(\rho)=\sgn(\sigma_S)\sgn(\sigma')\sgn(\tau').
\]
Summing over $\sigma'$ and $\tau'$ gives the tensor product of the two apartment cycles with coefficient $\sgn(\sigma_S)$; summing over $S$ gives the stated formula.
\end{proof}

\subsection{The Steinberg product and coproduct on \texorpdfstring{$\mc{H}$}{H}}\label{sec:induced}

Set $\Gamma_n=\GL_n(\Z)$. We now define the product and coproduct
\[
\mc H=\bigoplus_{n,k\geq 0}H_k(\Gamma_n;\St_n)
\]
to which the Quillen and sharbly structures will be compared.

In this section, all tensor products are over $\Q$. If $R_\bl\to \St_m$ and $S_\bl\to \St_n$ are flat resolutions over
$\Q[\Gamma_m]$ and $\Q[\Gamma_n]$, then $R_\bl\otimes S_\bl\to \St_m\otimes\St_n$ is a flat
$\Q[\Gamma_m\times\Gamma_n]$-resolution, and the usual K\"unneth map gives the ordinary cross product
\[
\times \colon H_i(\Gamma_m;\St_m)\otimes H_j(\Gamma_n;\St_n)
\longrightarrow H_{i+j}(\Gamma_m\times\Gamma_n;\St_m\otimes\St_n).
\]
Here, we must emphasize that we follow the Koszul sign convention, and the ordinary cross product treats $H_i(\Gamma_m; \St_m)$ as a graded object in degree $i$.  

However, in $\mc{H}$, the summand $H_i(\Gamma_m;\St_m)$ has total degree $m+i$. Thus, following the Koszul sign convention and accounting for the grading by total degree, we define  
\[
x\times_{\mathrm{tot}}y :=(-1)^{mj}(x\times y).
\]
The corresponding K\"unneth isomorphism is denoted $\kappa_{\mathrm{tot}}$.

\begin{lemma}\label{lem:kunneth-sign}
Let $T'_\bl$ and $T''_\bl$ be bounded-below complexes of flat modules over $\Q[G_1]$ and $\Q[G_2]$, with homology concentrated in degrees $a$ and $b$, equal there to $M'$ and $M''$. Then $T'_\bl \otimes T''_\bl$ is a bounded-below complex of flat $\Q[G_1\times G_2]$-modules with homology $M'\otimes M''$ concentrated in degree $a+b$, and $(T'_\bl\otimes T''_\bl)_{G_1\times G_2}=(T'_\bl)_{G_1}\otimes(T''_\bl)_{G_2}$. On the summand
\[
H_{a+i}\bigl((T'_\bl)_{G_1}\bigr)\otimes H_{b+j}\bigl((T''_\bl)_{G_2}\bigr)\;\cong\;H_i(G_1;M')\otimes H_j(G_2;M''),
\]
the K\"unneth isomorphism into $H_{a+b+i+j}\bigl((T'_\bl\otimes T''_\bl)_{G_1\times G_2}\bigr)\cong H_{i+j}(G_1\times G_2;M'\otimes M'')$ equals $(-1)^{aj}$ times the group-homology cross product.
\end{lemma}

\begin{proof}
Coinvariants over $G_1\times G_2$ split as $(T'_\bl\otimes T''_\bl)_{G_1\times G_2}=(T'_\bl)_{G_1}\otimes(T''_\bl)_{G_2}$, and over $\Q$ the displayed map is the algebraic K\"unneth isomorphism of this tensor product. Choose free resolutions $F_1\to\Q$ and $F_2\to\Q$ over $\Q[G_1]$ and $\Q[G_2]$; then $F_1\otimes F_2\to\Q$ is a free resolution over $\Q[G_1\times G_2]$. As in the proof of Lemma~\ref{lem:comparison}, the augmentations $(F_i\otimes T)_{G_i}\to T_{G_i}$ are quasi-isomorphisms by flatness and realize the identifications of Lemma~\ref{lem:comparison} on the bicomplexes $(F_i\otimes T)_{G_i}$, a class of degree $a_i+\ell$ being carried in $F_i$-degree $\ell$ and $T$-degree $a_i$. On these bicomplexes the cross product is induced by the interchange
\[
(F_1\otimes T'_\bl)\otimes(F_2\otimes T''_\bl)\xrightarrow{\ \chi\ }(F_1\otimes F_2)\otimes(T'_\bl\otimes T''_\bl)
\]
that moves the $T'_\bl$-factor past the $F_2$-factor, with the Koszul sign. On the summand of $F_1$-degree $i$, $T'_\bl$-degree $a$, $F_2$-degree $j$, and $T''_\bl$-degree $b$, this transposes a degree-$a$ factor past a degree-$j$ factor, contributing $(-1)^{aj}$. The group-homology cross product is the same interchange with coefficients in degree zero ($a=b=0$), where the sign is trivial; the two therefore differ by $(-1)^{aj}$.
\end{proof}

\subsubsection{Product}
Let $\alpha_{m,n}\colon \Gamma_m\times\Gamma_n\to \Gamma_{m+n}$ be the block diagonal inclusion. Let
\[
\mu_{m,n}\colon \St_m\otimes\St_n\longrightarrow \St_{m+n}
\]
be the product of Section~\ref{sec:product}. For
$x\in H_i(\Gamma_m;\St_m)$ and $y\in H_j(\Gamma_n;\St_n)$, define
\[
x\cdot y
=
(\alpha_{m,n},\mu_{m,n})_*(x\times_{\mathrm{tot}}y)
\in H_{i+j}(\Gamma_{m+n};\St_{m+n}).
\]
where $(\alpha_{m,n},\mu_{m,n})_*$ is the $(\alpha, f)_*$ defined in Section \ref{sec:comparison}.

\subsubsection{Coproduct}
Fix $0\leq d\leq n$. Let $U_d=\langle e_1,\ldots,e_d\rangle\subset \Q^n$. Let $P_d$ denote the corresponding parabolic subgroup  $P_d := \Stab_{\Gamma_n}(U_d)$, with $\pi_d \colon P_d \to L_d \cong \Gamma_d\times\Gamma_{n-d}$ its Levi quotient. Set 
$W_d := \St_d\otimes\St_{n-d}$, identifying $\Q^n/U_d$ with $\Q^{n-d}$ by the last coordinates. The action of $P_d$ on $W_d$ factors through the projection $\pi_d$. 

Transitivity of the $\Gamma_n$-action on rational $d$-planes gives a $\Q[\Gamma_n]$-module isomorphism
\[
\bigoplus_{d=0}^n \Ind_{P_d}^{\Gamma_n}W_d
\cong
\bigoplus_{U\subseteq \Q^n}\St(U)\otimes\St(\Q^n/U),
\qquad
g\otimes x\mapsto gx ,
\]
where $gx$ lies in the summand indexed by $gU_d$. Let
$\delta_d:\St_n\to \Ind_{P_d}^{\Gamma_n}W_d$ be the $d$th component of $\delta$ under this
identification, where $\delta$ is the map defined in Section \ref{sec:coproduct}. Set
\[
I_d:=\Ind_{P_d}^{\Gamma_n}W_d 
\quad \mbox{ and } \quad 
K_d(k):=
\bigoplus_{i+j=k}
H_i(\Gamma_d;\St_d)\otimes H_j(\Gamma_{n-d};\St_{n-d}).
\]
We also write
\[
\rho_d:=
(\pi_d,\mathrm{id})_*\circ \Sha_d^{-1}
:
H_k(\Gamma_n;I_d)\longrightarrow H_k(L_d;W_d),
\]
where $\Sha_d$ is Shapiro's isomorphism induced by
$P_d\hookrightarrow \Gamma_n$ and the coefficient map $x\mapsto 1\otimes x$.

Define $\Delta_d\colon 
H_k(\Gamma_n;\St_n)\longrightarrow K_d(k)$ to be the composite
\[
\Delta_d\colon H_k(\Gamma_n;\St_n)
\xrightarrow{\;(\mathrm{id},\delta_d)_*\;}
H_k(\Gamma_n;I_d)
\xrightarrow{\;\rho_d\;}
H_k(L_d;W_d)
\xrightarrow{\;\kappa_{\mathrm{tot}}^{-1}\;}
K_d(k).
\]
The Steinberg coproduct is then $\Delta := \sum_d \Delta_d$.

\begin{lemma}\label{lem:lift}
Let $X_\bl$ be a bounded-below complex of projective $\Q[\Gamma_n]$-modules with homology $\St_n$ concentrated in degree $N$, and let $T_\bl$ be a bounded-below complex of flat $\Q[L_d]$-modules with homology $W_d$ concentrated in degree $N$, regarded as a complex of $P_d$-modules via $\pi_d$. If $\tau\colon X_\bl \to \Ind_{P_d}^{\Gamma_n} T_\bl$ is $\Gamma_n$-equivariant and induces $\delta_d$ on $H_N$, then, under the canonical identification
\[
\bigl(\Ind_{P_d}^{\Gamma_n} T_\bl\bigr)_{\Gamma_n} \cong (T_\bl)_{L_d},
\]
the induced map $\tau_*$ on coinvariant homology is
\[
\rho_d \circ (\id,\delta_d)_* \colon H_k(\Gamma_n;\St_n) \lra H_k(L_d;W_d).
\]
In particular, $\kappa_{\mathrm{tot}}^{-1}\circ\tau_*=\Delta_d$.
\end{lemma}

\begin{proof}
Let $R_\bl$ be a bounded-below complex of projective $\Q[P_d]$-modules with homology $W_d$ concentrated in degree $N$ --- for instance, the $N$-fold shift of a projective $\Q[P_d]$-resolution of $W_d$. Since $\Q[\Gamma_n]$ is free as a right $\Q[P_d]$-module, induction from $P_d$ to $\Gamma_n$ is exact and preserves projectives, so $\Ind_{P_d}^{\Gamma_n}R_\bl$ and $\Ind_{P_d}^{\Gamma_n}T_\bl$ have homology $\Ind_{P_d}^{\Gamma_n}W_d$ concentrated in degree $N$, the former by projective modules. By Lemma~\ref{lem:chain-lifting}, applied in degree $N$, there are equivariant lifts
\[
\Lambda\colon X_\bl \longrightarrow \Ind_{P_d}^{\Gamma_n}R_\bl,
\qquad
\widetilde\pi\colon R_\bl \longrightarrow T_\bl
\]
of $\delta_d$ and $\id_{W_d}$, respectively, with $\widetilde\pi$ $\pi_d$-equivariant.

Both $\tau$ and $\Ind_{P_d}^{\Gamma_n}(\widetilde\pi)\circ\Lambda$ are $\Gamma_n$-equivariant lifts of $\delta_d$ to $\Ind_{P_d}^{\Gamma_n}T_\bl$. By Lemma~\ref{lem:chain-lifting} they are equivariantly chain homotopic, and therefore induce the same map on coinvariant homology. On coinvariants, the latter composite is
\[
(X_\bl)_{\Gamma_n}
\xrightarrow{\;\Lambda\;}
\bigl(\Ind_{P_d}^{\Gamma_n}R_\bl\bigr)_{\Gamma_n}
\xrightarrow{\;\cong\;}
(R_\bl)_{P_d}
\xrightarrow{\;\widetilde\pi\;}
(T_\bl)_{L_d}.
\]
Lemma~\ref{lem:comparison}, in degree $N$, identifies the first and last maps on homology with $(\id,\delta_d)_*$ and $(\pi_d,\id)_*$, respectively. The middle isomorphism induces $\Sha_d^{-1}$, since $\Sha_d$ is represented by $r\mapsto 1\otimes r$. Hence $\tau$ induces
\[
(\pi_d,\id)_*\circ\Sha_d^{-1}\circ(\id,\delta_d)_*
=\rho_d\circ(\id,\delta_d)_*,
\]
as claimed.
\end{proof}

\subsubsection{Unit and counit}
The unit is the inclusion of the summand $H_0(\Gamma_0;\St_0)=\Q$, and the counit is the
projection onto it. Since $\St_0=\Q$, products with the rank-zero summand are the evident
identifications. By Proposition~\ref{prop:coproduct}, the $U=0$ and $U=\Q^n$ components of
$\delta$ are the canonical maps $\St_n\cong \St_0\otimes\St_n$ and
$\St_n\cong \St_n\otimes\St_0$. For $n=0$ we define
$\delta_0:\St_0\to \St_0\otimes\St_0$ to be the canonical isomorphism. These data define the Steinberg product and coproduct on $\mc H$. The bialgebra axioms are not verified here; they
follow a posteriori from the comparisons with the Hopf algebra structures of
Ash--Miller--Patzt and Brown--Chan--Galatius--Payne.

\begin{remark}\label{rem:duoidal}
The pair $(\mu, \delta)$ exists already on the Steinberg modules, before passing to group homology. The modules $\St$ form a representation of the groupoid $\coprod_n \GL_n(\Z)$, and the category of such representations carries two monoidal structures: Day convolution \cite{miller2020stability}, for which an algebra is a family of $\GL_m(\Z) \times \GL_n(\Z)$-equivariant maps $\St_m \otimes \St_n \to \St_{m+n}$, and parabolic induction \cite{nagpal2019vi}, for which a coalgebra is a family of $\GL_n(\Z)$-equivariant maps
\[ \St_n \lra \bigoplus_{d=0}^{n} \Ind_{P_d}^{\GL_n(\Z)} \bigl( \St_d \otimes \St_{n-d} \bigr), \]
with $P_d$ acting through its Levi quotient. The product $\mu$ is such a family, and so, under the identification above, is the coproduct $\delta$. Ash--Miller--Patzt show that the two structures are compatible and that $\St$ is a bialgebra in the resulting duoidal category \cite[Remark~1.4]{ash2024hopf}. The comparisons of Sections~\ref{sec:amp} and~\ref{sec:bcgp} then identify the Hopf algebras of \cite{ash2024hopf} and \cite{brown2024hopf} with a homological shadow of this bialgebra.
\end{remark}

\section{Comparing with the sharbly Hopf structure}\label{sec:amp}

Let $V$ be a finitely generated free $\Z$-module of rank $n$. Following \cite[Section 2]{ash2024hopf}, let $X(V)$ be the simplicial complex whose vertices are the nonzero vectors of $V$ taken up to sign, and in which every finite set of vertices is a simplex, and let $L(V)$ be the subcomplex of simplices whose vertices do not span $V \otimes \Q$. For $k \geq 0$ set
\[ \Sh_k(V) \;=\; C_{n+k-1}\bigl(X(V), L(V);\, \Q\bigr), \]
the relative simplicial chains: the $\Q$-vector space on the symbols $[v_1, \dots, v_{n+k}]$, where the $v_i$ are distinct vertices spanning $V \otimes \Q$, subject to $[v_{\sigma(1)}, \dots, v_{\sigma(n+k)}] = \sgn(\sigma)\,[v_1, \dots, v_{n+k}]$, with the simplicial boundary as differential and the augmentation $\Sh_0(V) \to \St(V \otimes \Q)$ sending $[v_1, \dots, v_n]$ to the apartment class of the basis $v_1, \dots, v_n$ of $V \otimes \Q$. We write $\Sh_\bl(V) \to \St(V \otimes \Q)$ for the augmented complex and abbreviate $\St_n = \St(\Z^n \otimes \Q)$ as before. For the rank-zero module we set $\Sh(0) = \Q$, supported in degree 0, with augmentation the identity map to $\St_0 = \Q$.

\begin{prop}\label{prop:sharbly}
The map $\Sh_\bl(\Z^n) \to \St_n$ is a resolution of $\St_n$ by projective $\Q[\GL_n(\Z)]$-modules.
\end{prop}

\begin{proof}
For exactness, see \cite[Section 2]{ash2024hopf}. For projectivity, decompose $\Sh_k(\Z^n)$ over the $\GL_n(\Z)$-orbits of generators. The setwise stabilizer of the unoriented simplex corresponding to a generator $[v_1, \ldots, v_{n+k}]$ is finite and acts by the sign of the induced permutation on the set of pairs $\{\{ \pm v_1 \}, \ldots, \{ \pm v_{n+k} \} \}$. Each orbit summand is thus induced from a character of a finite group. Characters of finite groups are projective over $\Q$, and induction preserves projectivity, so $\Sh_k(\Z^n)$ is projective.
\end{proof}

Ash--Miller--Patzt equip the sharbly resolutions with a product and a coproduct \cite[Definitions 2.1 and 2.6--2.7]{ash2024hopf}. Let $\phi \colon \Q^m \to \Q^{m+n}$ and $\psi \colon \Q^n \to \Q^{m+n}$ be the embeddings given by the first $m$ and last $n$ coordinates, respectively, as in  Section~\ref{sec:product}. The product is the chain map 
\begin{equation*}
\begin{aligned}
\nabla\colon\; \Sh_k(\Z^m) \otimes \Sh_\ell(\Z^n) &\lra \Sh_{k+\ell}(\Z^{m+n}), \\
[v_\bl] \otimes [w_\bl] &\longmapsto [\phi(v_1), \dots, \phi(v_{m+k}),\, \psi(w_1), \dots, \psi(w_{n+\ell})].
\end{aligned}
\end{equation*}
It is equivariant with respect to the block diagonal embedding. 

For a saturated submodule $L \subseteq \Z^n$ write $S(L) := \{i : v_i \in L \otimes \Q\}$. The coproduct is 
\[
\begin{aligned}
\Delta\colon\; \Sh_k(\Z^n) &\lra \bigoplus_{L}\, \bigoplus_{s+t=k} \Sh_s(L) \otimes \Sh_t(\Z^n/L), \\
[v_1, \dots, v_{n+k}] &\longmapsto \sum_{L} \sgn(\sigma_L)\, [v_{i_1}, \dots, v_{i_p}] \otimes [\bar v_{j_1}, \dots, \bar v_{j_{n+k-p}}],
\end{aligned}
\]
where the sum runs over the saturated $L$ for which $S(L) = \{i_1 < \cdots < i_p\}$ spans $L \otimes \Q$, the complementary indices are $j_1 < \cdots < j_{n+k-p}$, bars denote images in $\Z^n/L$, and $\sigma_L$ is the shuffle $(i_1, \dots, i_p, j_1, \dots, j_{n+k-p})$. A term with a zero quotient vertex or with two quotient vertices equal up to sign is understood to be zero. Only finitely many $L$ contribute. 

Passing to coinvariants and homology, the product is obtained by applying $\kappa_{\mathrm{tot}}$, the map induced by $\nabla$, and change of groups along the block diagonal $\GL_m(\Z)\times\GL_n(\Z)\to\GL_{m+n}(\Z)$. The coproduct is obtained by applying the map induced by $\Delta$, the canonical identification
\begin{equation}\label{eq:sharbly-coinv}
\Bigl(\, \bigoplus_{L} \Sh_\bl(L) \otimes \Sh_\bl(\Z^n/L) \Bigr)_{\GL_n(\Z)} \;\cong\; \bigoplus_{d=0}^{n} \bigl(\Sh_\bl(\Z^d) \otimes \Sh_\bl(\Z^{n-d})\bigr)_{\GL_d(\Z) \times \GL_{n-d}(\Z)},
\end{equation}
and $\kappa_{\mathrm{tot}}^{-1}$ \cite[Section~2]{ash2024hopf}. The identification \eqref{eq:sharbly-coinv} is the induced-module identification of Section~\ref{sec:induced}; the unipotent radical acts trivially \cite[Lemma~2.8]{ash2024hopf}. The unit is the inclusion of the bidegree-$(0,0)$ summand, and the counit is its projection. The differential, Leibniz rules, and braiding in \cite{ash2024hopf} use total degree and therefore agree with our Koszul sign convention \cite[Lemmas~2.5, 2.10, and 2.11]{ash2024hopf}.

\begin{theorem}\label{thm:amp}
The sharbly product and coproduct on $\mc{H}$ coincide with the Steinberg product and coproduct.
\end{theorem}

\begin{proof}
Let $\Theta$ be the automorphism of $\bigoplus_{a,b} \Sh_k(\Z^a) \otimes \Sh_\ell(\Z^b)$ whose restriction to $\Sh_k(\Z^a) \otimes \Sh_\ell(\Z^b)$ is multiplication by $(-1)^{a\ell}$.  Composing with this automorphism converts the differential of the Ash--Miller--Patzt differential graded bialgebra, with Koszul signs governed by total degree, to that of the ordinary tensor product of resolutions.  

Note, in particular, that $\Theta$ is the identity on $\Sh_0$. Thus $\nabla \circ \Theta$ and $\Theta \circ \Delta$ are equivariant chain maps of the underlying resolutions, agreeing with $\nabla$ and $\Delta$ in degree zero. It remains to identify their degree-zero augmentations and apply the comparison lemmas of Sections~\ref{sec:comparison} and~\ref{sec:induced}.

For the product, let $\alpha_{m,n}\colon \GL_m(\Z) \times \GL_n(\Z) \to \GL_{m+n}(\Z)$ be the block diagonal embedding. By Proposition~\ref{prop:sharbly}, $\Sh_\bl(\Z^m) \otimes \Sh_\bl(\Z^n) \to \St_m \otimes \St_n$ is a projective resolution over $\Q[\GL_m(\Z) \times \GL_n(\Z)]$, and $\nabla \circ \Theta$ is an $\alpha_{m,n}$-equivariant lift of $\mu_{m,n}$, since in degree zero it sends $[v_1, \dots, v_m] \otimes [w_1, \dots, w_n]$ to $\llb \phi(v_1), \dots, \phi(v_m), \psi(w_1), \dots, \psi(w_n) \rrb$ by Proposition~\ref{prop:product}. By Lemma~\ref{lem:comparison} it induces $(\alpha_{m,n}, \mu_{m,n})_*$ on coinvariants, so the sharbly product is $(\alpha_{m,n}, \mu_{m,n})_* \circ \times_{\mathrm{tot}}$, the Steinberg product.

For the coproduct, fix $0 \leq d \leq n$ and set $T_\bl:=\Sh_\bl(\Z^d)\otimes\Sh_\bl(\Z^{n-d})$. Let $\Delta^{\langle d \rangle}\colon \Sh_\bl(\Z^n) \to \Ind_{P_d}^{\GL_n(\Z)}T_\bl$ be the rank-$d$ component of $\Theta \circ \Delta$ \cite[Section~2]{ash2024hopf}. By Proposition~\ref{prop:sharbly}, the source is a complex of projective $\Q[\GL_n(\Z)]$-modules and $T_\bl\to W_d$ is a projective, hence flat, $\Q[L_d]$-resolution; we regard $T_\bl$ as a $P_d$-complex through $\pi_d$. The map $\Delta^{\langle d \rangle}$ lifts $\delta_d$, since in degree zero it is the $|S|=d$ part of Proposition~\ref{prop:coproduct}. By Lemma~\ref{lem:lift}, $\kappa_{\mathrm{tot}}^{-1} \circ (\Delta^{\langle d \rangle})_* = \Delta_d$; summing over $d$, the sharbly coproduct is the Steinberg coproduct.
\end{proof}

\section{Comparing with the Quillen Hopf structure}\label{sec:bcgp}
We recall the part of the Waldhausen construction used below. Let $\mathrm{Proj}_\Z$ be the category of finitely generated projective $\Z$-modules. For $p\geq 0$, let $S_p(\mathrm{Proj}_\Z)$ be the groupoid whose objects are the diagrams
\[
\begin{aligned}
\xymatrix{0 \ar[r] & Q_{0,1} \ar[d] \ar[r] & Q_{0,2} \ar[d] \ar[r] & \cdots \ar[d] \ar[r] & Q_{0,p} \ar[d] \\
 & 0 \ar[r] & Q_{1,2} \ar[d] \ar[r] & \cdots \ar[d] \ar[r] & Q_{1,p} \ar[d] \\
 & & 0 \ar[r] & \ddots \ar[d] \ar[r] & \vdots \ar[d] & \\
 & & & 0 \ar[r] & Q_{p-1,p} \ar[d] \\
 & & & & 0,}
\end{aligned}
\]
where each horizontal arrow is an injection with projective cokernel, each vertical arrow is a quotient map, and $Q_{i,j}$ is a chosen cokernel of $Q_{0,i}\to Q_{0,j}$ (these are the $P_{i,j}$ of \cite{brown2024hopf}); morphisms are isomorphisms of diagrams. The face and degeneracy functors delete and duplicate the row and column of a given index, so the nerves form a bisimplicial set
\[
X_{p,q}:=N_qS_p(\mathrm{Proj}_\Z),
\qquad
\mathrm{BK}(\Z):=|X|,
\]
with $K_i(\Z)=\pi_{i+1}\mathrm{BK}(\Z)$ \cite[Section~2.2]{brown2024hopf}.

The rank of a diagram is $\rk Q_{0,p}$. Diagrams of rank at most $n$ form a bisimplicial subset $F_nX$, giving the exhaustive rank filtration $F_n\mathrm{BK}(\Z):=|F_nX|$; set $F_{-1}X=\varnothing$. We write $\EQ^*_{*,*}$ for its rational homology spectral sequence, so that
\[
\EQ^1_{n,k}
=
H_{n+k}\bigl(F_n\mathrm{BK}(\Z),F_{n-1}\mathrm{BK}(\Z);\Q\bigr)
\;\cong\;
H_k\bigl(\GL_n(\Z);\St_n\bigr),
\]
the latter isomorphism fixed in Section~\ref{sec:rigidified}.

Direct sum of diagrams defines a filtered product
\[
\m\colon X\times X\longrightarrow X,
\qquad
\bigl((Q_{i,j}),(Q'_{i,j})\bigr)\longmapsto\bigl(Q_{i,j}\oplus Q'_{i,j}\bigr),
\]
carrying $F_mX\times F_nX$ into $F_{m+n}X$. The coproduct uses edgewise subdivision in the $S_\bl$-direction: the natural transformation $\Phi\colon\es(X)\to X\times X$ sends a diagram with top row of length $2p+1$ to its front truncation and its row of quotients,
\begin{equation} \label{eq:Phi}
\begin{aligned}
\Phi\colon\quad 0\to Q_{0,1}\to\cdots\to Q_{0,2p+1}
\longmapsto
\bigl(&0\to Q_{0,1}\to\cdots\to Q_{0,p},\\
&0\to Q_{p+1,p+2}\to\cdots\to Q_{p+1,2p+1}\bigr).
\end{aligned}
\end{equation}
Indeed,
\[
\rk Q_{0,p}+\rk Q_{p+1,2p+1}
=
\rk Q_{0,p}+\rk Q_{0,2p+1}-\rk Q_{0,p+1}
\leq \rk Q_{0,2p+1},
\]
so $\Phi$ is filtered for the product filtration. Composing its realization with the filtered inverse of the natural homeomorphism $|\es(X)|\cong|X|$ gives a filtered map $\mathrm{BK}(\Z)\to\mathrm{BK}(\Z)\times\mathrm{BK}(\Z)$.

Through the K\"unneth isomorphism these maps induce a product and coproduct on every page of $\EQ^*_{*,*}$; with the unit and counit in bidegree $(0,0)$, this is the Hopf algebra structure of \cite[Theorem~3.18]{brown2024hopf}, defined there space-theoretically. We compute its operations on $\EQ^1_{*,*}$ from the filtered maps $\m$ and $\Phi$ alone, through chain-level models on the totalized complexes of Sections~\ref{sec:bcgp-product} and~\ref{sec:bcgp-coproduct}, and identify them with the Steinberg product and coproduct.

\subsection{A rigidified model}\label{sec:rigidified}
We now make the identification $\EQ^1_{n,k}\cong H_k(\GL_n(\Z);\St_n)$ explicit. Fix $n \geq 0$, and write $S^n_p \subseteq S_p(\mathrm{Proj}_\Z)$ for the full subgroupoid of diagrams with $Q_{0,p} \cong \Z^n$. In bidegree $(p,q)$, the quotient $F_nX/F_{n-1}X$ is the pointed set $N_q S^n_p \sqcup \{\ast\}$. Let $C^n$ be its double complex of reduced chains, with the conventions of Section~\ref{sec:chains}. For $p > 0$, the face map $d_0$ sends a diagram to the basepoint if $Q_{0,1} \neq 0$, and $d_p$ does so if $Q_{p-1,p} \neq 0$. By the Eilenberg--Zilber theorem for bisimplicial objects \cite[Theorem~8.5.1]{weibel1994introduction}, $H_{n+k}\bigl(\mathrm{Tot}(C^n)\bigr) = \EQ^1_{n,k}$.

Let $\cT_p(\Z^n)$ be the groupoid of pairs $((Q_{i,j}),\beta)$ consisting of an object of $S^n_p$ and a \emph{rigidification} $\beta\colon Q_{0,p} \xrightarrow{\sim} \Z^n$. Its morphisms are isomorphisms of diagrams commuting with the rigidifications. The group $\GL_n(\Z)$ acts by postcomposing $\beta$. We lift the face and degeneracy maps of $F_nX/F_{n-1}X$ to the pointed sets $N_q\cT_p(\Z^n) \sqcup \{\ast\}$, as follows. 

For fixed $q$, the lifted face map $d_i$, for $0 < i < p$, acts on a string of isomorphisms in $N_q\cT_p(\Z^n)$ by deleting row and column $i$ from each diagram. The degeneracy maps repeat a row and column. These operations leave $Q_{0,p}$ and $\beta$ unchanged. The face maps $d_0$ and $d_p$ send a string to the basepoint when they lower the rank of the diagram; otherwise $\beta$ induces an isomorphism from the new terminal module to $\Z^n$. Let $D^n$ be the double complex of reduced chains on the resulting pointed bisimplicial set. Each $D^n_{p,q}$ is a free $\Q[\GL_n(\Z)]$-module, and forgetting $\beta$ induces an isomorphism $(D^n)_{\GL_n(\Z)} \cong C^n$.
 
To each object $((Q_{i,j}),\beta)$ in $\cT_p(\Z^n)$ we associate the flag $0 = U_0 \subseteq \cdots \subseteq U_p = \Q^n$, where $U_i$ is the rational span of the image of $Q_{0,i}$ in $\Z^n$. Define
\[
c_n\colon \mathrm{Tot}(D^n) \longrightarrow \wtl C_*(\Sigma\ST_n)
\]
by sending the class of $((Q_{i,j}),\beta)$ in $D^n_{p,0}$ to the class of its associated flag, and setting $c_n = 0$ on $D^n_{p,q}$ for $q > 0$.

\begin{lemma}\label{lem:rigid}
The map $c_n\colon \mathrm{Tot}(D^n) \to \wtl C_*(\Sigma\ST_n)$ is a $\GL_n(\Z)$-equivariant quasi-isomorphism. In particular, $H_*(\mathrm{Tot}(D^n)) \cong \St_n$, concentrated in degree $n$.
\end{lemma}

\begin{proof}
A flag $0 = U_0 \subseteq \cdots \subseteq U_p = \Q^n$ determines a rigidified diagram with $Q_{i,j} = (U_j \cap \Z^n)/(U_i \cap \Z^n)$. Every object of $\cT_p(\Z^n)$ with this flag is uniquely isomorphic to this diagram, since each $Q_{0,i}$ has saturated image in $\Z^n$. Thus the functor $f_p$ assigning to an object its flag is an equivalence from $\cT_p(\Z^n)$ to the discrete category of such flags.

After adjoining basepoints, the functors $f_p$ commute with the face and degeneracy maps in the $p$-direction. For $0 < i < p$, the face map $d_i$ deletes $U_i$ from the flag, and the degeneracy maps repeat a subspace. For $p > 0$, the face map $d_0$ sends a flag to the basepoint exactly when $U_1 \neq 0$, and $d_p$ does so exactly when $U_{p-1} \neq \Q^n$, as in $\Sigma\ST_n$. The induced map on chains is $c_n$, since the nerve of a discrete category has no nondegenerate simplices in positive degrees. Replacing $\beta$ by $g\beta$ replaces each $U_i$ by $gU_i$, so $c_n$ is $\GL_n(\Z)$-equivariant.

For each $p$, the equivalence of groupoids induces an isomorphism on the homology of their nerves. Taking reduced chains on each nerve and then normalizing in the $p$-direction preserves this property, since normalization is exact. Hence the map of double complexes inducing $c_n$ is a quasi-isomorphism in each column, and $c_n$ is a quasi-isomorphism by \cite[Lemma~2.7.3]{weibel1994introduction}. The last assertion follows from the homology of $\Sigma\ST_n$ recalled in Section~\ref{sec:tools}.
\end{proof}

Since $\mathrm{Tot}(D^n)$ is a bounded-below complex of free $\Q[\GL_n(\Z)]$-modules with homology $\St_n$ concentrated in degree $n$, Lemma~\ref{lem:comparison} and the identification $(D^n)_{\GL_n(\Z)} \cong C^n$ give
\begin{equation} \label{eq:identification}
\EQ^1_{n,k} = H_{n+k}\bigl(\mathrm{Tot}(C^n)\bigr) = H_{n+k}\bigl((\mathrm{Tot} D^n)_{\GL_n(\Z)}\bigr) \cong H_k\bigl(\GL_n(\Z);\St_n\bigr).
\end{equation}
All subsequent identifications of $\EQ^1$ with $\mc H$ use this isomorphism. We now compare it with the identification in \cite[Proposition~2.15]{brown2024hopf}. Both constructions use the same rigidified groupoids. For $n \geq 1$, the relation between our model using the doubly suspended building and the Tits-building chains used in that proof is given by deleting the endpoints of a flag:
\[
(0 \subset U_1 \subset \cdots \subset U_{p-1} \subset \Q^n)
\longmapsto
(U_1 \subset \cdots \subset U_{p-1}).
\]
For strict flags, this identifies the generators of $\wtl C_p(\Sigma\ST_n)$ with those of $\wtl C_{p-2}(\T(\Q^n))$.

We compare the differentials under this correspondence. In each complex, the differential is the alternating sum of the faces obtained by deleting a vertex, with vertices numbered starting at zero. In the full flag, $U_i$ corresponds to the $i$th vertex, so its deletion appears with coefficient $(-1)^i$. In the Tits-building flag, $U_i$ is the vertex labeled by $i-1$, so its deletion appears with coefficient $(-1)^{i-1}$. The two remaining faces of the full flag, obtained by deleting $0$ or $\Q^n$, vanish in the quotient defining $\Sigma\ST_n$. Consequently, after identifying the chain groups by the displayed correspondence, the two differentials are negatives of one another.

Changing the sign of a differential changes neither its kernel nor its image. The correspondence therefore identifies cycles and boundaries, and hence homology, with degrees differing by two. Moreover, deleting the endpoints sends each apartment cycle in $\Sigma\ST_n$ to the corresponding apartment cycle in the Tits building, with the same coefficients. Thus the two constructions use the same identification with $\St_n$, and hence give the same isomorphism \eqref{eq:identification}. 

\subsection{Comparison of the products}\label{sec:bcgp-product}

The direct-sum map lifts naturally to the rigidified complexes. Under the collapse of Lemma~\ref{lem:rigid}, this lift becomes the flag-sum map defining $\mu_{m,n}$. We now compare the induced products, keeping track of the K\"unneth sign arising from the degree shifts.

\begin{prop}\label{prop:bcgp-product}
Under the identification $\EQ^1_{n,k}\cong H_k(\GL_n(\Z);\St_n)$ fixed in \eqref{eq:identification}, the product induced by $\m$ agrees with the Steinberg product. In particular, for $x$ in $H_j(\GL_m(\Z);\St_m)$ and $y$ in $H_k(\GL_n(\Z);\St_n)$, it is given by
\[
x\cdot y=(\alpha_{m,n},\mu_{m,n})_*\bigl(x\times_{\mathrm{tot}}y\bigr).
\]
\end{prop}

\begin{proof}
Direct sum lifts to the rigidified groupoids $\cT_p(\Z^m)\times\cT_p(\Z^n)\longrightarrow\cT_p(\Z^{m+n})$ via
\[
\bigl(((Q_{i,j}),\beta),((Q'_{i,j}),\beta')\bigr)\longmapsto\bigl((Q_{i,j}\oplus Q'_{i,j}),\beta\oplus\beta'\bigr).
\]
This map is equivariant for the block diagonal inclusion $\alpha_{m,n}\colon\GL_m(\Z)\times\GL_n(\Z)\to\GL_{m+n}(\Z)$ and commutes with the simplicial structure and the outer-face conventions. Applying the Eilenberg--Zilber map in the two simplicial directions and totalizing gives an $\alpha_{m,n}$-equivariant chain map
\[
\tau_{m,n}\colon\mathrm{Tot}(D^m)\otimes\mathrm{Tot}(D^n)\longrightarrow\mathrm{Tot}(D^{m+n}).
\]
After passage to coinvariants, this is the map induced by the filtered direct sum $\m$ following the Eilenberg--Zilber map. Hence the product on $\EQ^1$ is the K\"unneth map followed by $(\tau_{m,n})_*$.

The maps $c_r$ commute with direct sum and with the shuffle maps, so the following diagram commutes:
\[
\xymatrix@C=3.5em{
\mathrm{Tot}(D^m)\otimes\mathrm{Tot}(D^n)
    \ar[rr]^-{\tau_{m,n}}
    \ar[d]_{c_m\otimes c_n}
&&
\mathrm{Tot}(D^{m+n})
    \ar[d]^{c_{m+n}}
\\
\wtl C_*(\Sigma\ST_m)\otimes\wtl C_*(\Sigma\ST_n)
    \ar[rr]_-{(\theta_{m,n})_*\circ\mathrm{EZ}}
&&
\wtl C_*(\Sigma\ST_{m+n}).
}
\]
Indeed, a pair of rigidified diagrams collapses to flags $U_\bl$ and $V_\bl$, while their direct sum collapses to the flag $\phi(U_\bl)+\psi(V_\bl)$. The bottom map is therefore the chain map defining $\mu_{m,n}$ in Section~\ref{sec:product}. Under the identifications of Lemma~\ref{lem:rigid}, it follows that $H_{m+n}(\tau_{m,n})=\mu_{m,n}$.

The complexes in the source and target of $\tau_{m,n}$ are bounded below and free over the corresponding group rings, with homology concentrated in degree $m+n$. Lemma~\ref{lem:comparison} therefore identifies the map induced by $\tau_{m,n}$ on coinvariant homology with $(\alpha_{m,n},\mu_{m,n})_*$.

It remains to compare the K\"unneth conventions. By Lemma~\ref{lem:kunneth-sign} with $a=m$ and $b=n$, on the summand of group-homological degrees $(j,k)$ the K\"unneth map of the coinvariant complexes is $(-1)^{mk}$ times the ordinary cross product, which is $\times_{\mathrm{tot}}$ by the definition in Section~\ref{sec:induced}. Consequently, the product on $\EQ^1$ is $(\alpha_{m,n},\mu_{m,n})_*\circ\times_{\mathrm{tot}}$.
\end{proof}

\subsection{Comparison of the coproducts}\label{sec:bcgp-coproduct}

On the associated graded, the map $\Phi$ given by \eqref{eq:Phi} is supported on the diagrams for which the middle map $Q_{0,p}\to Q_{0,p+1}$ is an isomorphism, equivalently $Q_{p,p+1}=0$. For such a diagram, the rigidification identifies the common image of the two middle terms with a saturated submodule $L \subseteq \Z^n$, and the two factors are the front diagram over $L$ and the quotient diagram over $\Z^n/L$. Under the collapse of Lemma~\ref{lem:rigid}, this is the splitting map $r$ defining the Steinberg coproduct.

For a finitely generated free $\Z$-module $V$, let $D(V)$ denote the rigidified complex defined in Section~\ref{sec:rigidified}, with terminal rigidification to $V$; thus $D(\Z^n) = D^n$, and Lemma~\ref{lem:rigid} holds with $\Z^n$, $\GL_n(\Z)$, $\Sigma\ST_n$ replaced by $V$, $\GL(V)$, $\Sigma\ST(V \otimes \Q)$ respectively. Write $\es(D^n)$ for the double complex obtained by applying edgewise subdivision in the first simplicial direction before taking chains. Fix $0 \leq d \leq n$, and let $P_d \leq \GL_n(\Z)$ be the stabilizer of $\Z^d \subseteq \Z^n$, with Levi quotient $L_d \cong \GL_d(\Z) \times \GL_{n-d}(\Z)$, as in Section~\ref{sec:induced}. The action of $\GL_n(\Z)$ on saturated rank-$d$ submodules gives, via $g \otimes x \mapsto gx$, a canonical isomorphism of complexes of $\Q[\GL_n(\Z)]$-modules
\[
\Ind_{P_d}^{\GL_n(\Z)}\bigl(\mathrm{Tot}(D^d) \otimes \mathrm{Tot}(D^{n-d})\bigr) \xrightarrow{\;\cong\;} \bigoplus_{\substack{L \subseteq \Z^n\ \text{saturated} \\ \rk L = d}} \mathrm{Tot}(D(L)) \otimes \mathrm{Tot}(D(\Z^n/L)),
\]
where $P_d$ acts on the tensor product through its Levi quotient $\pi_d$.

Let a rigidified simplex of $\es(D^n)$ have top row
\[
0 \longrightarrow Q_{0,1} \longrightarrow \cdots \longrightarrow Q_{0,2p+1}
\]
and rigidification $\beta \colon Q_{0,2p+1} \xrightarrow{\sim} \Z^n$. Call it \emph{balanced} if its middle map is an isomorphism, equivalently if $Q_{p,p+1}=0$, and set its image to zero otherwise. In the balanced case, let $L \subseteq \Z^n$ be the image of $Q_{0,p}$ under $\beta$ and send the simplex to the pair consisting of
\[
0 \longrightarrow Q_{0,1} \longrightarrow \cdots \longrightarrow Q_{0,p}, \qquad 0 \longrightarrow Q_{p+1,p+2} \longrightarrow \cdots \longrightarrow Q_{p+1,2p+1},
\]
rigidified by $\beta|_{Q_{0,p}} \colon Q_{0,p} \xrightarrow{\sim} L$ and by the induced isomorphism $Q_{p+1,2p+1} \xrightarrow{\sim} \Z^n/L$. Applying reduced chains and the Alexander--Whitney maps in both simplicial directions gives a $\GL_n(\Z)$-equivariant chain map
\[
\widetilde\Phi \colon \mathrm{Tot}(\es(D^n)) \longrightarrow \bigoplus_{L \subseteq \Z^n\ \text{saturated}} \mathrm{Tot}(D(L)) \otimes \mathrm{Tot}(D(\Z^n/L)).
\]

\begin{lemma}\label{lem:phitilde}
Forgetting the rigidifications, $\widetilde\Phi$ is the Alexander--Whitney map following the chain map that $\Phi$ induces on the associated graded of the rank filtration. Under the collapse of Lemma~\ref{lem:rigid}, it covers $\mathrm{AW} \circ r$, with $r$ the splitting map of Section~\ref{sec:coproduct}.
\end{lemma}

\begin{proof}
By the rank inequality from the start of this section, when the terminal module has rank $n$ the two factors produced by $\Phi$ have total rank at most $n$, with equality exactly when the middle map $Q_{0,p}\to Q_{0,p+1}$ is an isomorphism; the unbalanced diagrams therefore map to the basepoint of the associated graded. The assignment is $\GL_n(\Z)$-equivariant, $g$ carrying the $L$-summand to the $gL$-summand, and compatible with faces and degeneracies: the inner structure maps act factorwise, no structure map creates balance, and the two outer faces survive precisely when the corresponding outer faces of both factors survive, the rigidifications transported as in Section~\ref{sec:rigidified}. Forgetting $\beta$ thus recovers the chain map induced by $\Phi$ on the associated graded, followed by the Alexander--Whitney maps. After collapse, a balanced diagram becomes a flag with common middle term $L \otimes \Q$, sent to its front and quotient flags --- the defining formula for $r_{L \otimes \Q}$ in Section~\ref{sec:coproduct} --- so $\widetilde\Phi$ covers $\mathrm{AW} \circ r$.
\end{proof}

\begin{prop}\label{prop:bcgp-coproduct}
Under the identification $\EQ^1_{n,k} \cong H_k(\GL_n(\Z);\St_n)$ fixed in \eqref{eq:identification}, the coproduct induced by $\Phi$ agrees with the Steinberg coproduct of Section~\ref{sec:induced}. In particular, its rank-$d$ component is $\Delta_d$.
\end{prop}

\begin{proof}
Fix $d$ and put
\[
X_\bl := \mathrm{Tot}(\es(D^n)), \qquad T_\bl := \mathrm{Tot}(D^d) \otimes \mathrm{Tot}(D^{n-d}), \qquad W_d := \St_d \otimes \St_{n-d}.
\]
Let
\[
\tau_d \colon X_\bl \longrightarrow \Ind_{P_d}^{\GL_n(\Z)} T_\bl
\]
be the rank-$d$ component of $\widetilde\Phi$, followed by the inverse of the displayed induced-module isomorphism. By Lemma~\ref{lem:rigid} applied to the edgewise subdivision, $X_\bl$ is a bounded-below complex of free $\Q[\GL_n(\Z)]$-modules with homology $\St_n$ concentrated in degree $n$; we identify $H_n(X_\bl)$ with $\St_n$ by the collapse of Lemma~\ref{lem:rigid} followed by $(\varphi_0)_*$, both quasi-isomorphisms. Likewise, by Lemma~\ref{lem:rigid} and the K\"unneth theorem, $T_\bl$ is a bounded-below complex of free $\Q[L_d]$-modules with homology $W_d$ concentrated in degree $n$. By Lemma~\ref{lem:phitilde} and the definition of $\delta$ in Section~\ref{sec:coproduct}, $\tau_d$ lifts $\delta_d$: that is, $H_n(\tau_d) = \delta_d$ under these identifications.

By Lemma~\ref{lem:lift} with $N = n$, the map induced by $\tau_d$ on coinvariant homology is, through $(\varphi_0)_*$,
\[
\rho_d \circ (\id, \delta_d)_* \colon H_k(\GL_n(\Z);\St_n) \longrightarrow H_k(L_d; W_d).
\]
The rank-$d$ component of the coproduct on $\EQ^1$ is this composite followed by the inverse K\"unneth map under the canonical identification
\[
(T_\bl)_{L_d} \cong \mathrm{Tot}(C^d) \otimes \mathrm{Tot}(C^{n-d}).
\]
By Lemma~\ref{lem:kunneth-sign} with $a=d$ and $b=n-d$, this K\"unneth map is, on the summand of degrees $(i,j)$, the sign $(-1)^{dj}$ times the ordinary cross product, hence $\times_{\mathrm{tot}}$; its inverse is therefore $\kappa_{\mathrm{tot}}^{-1}$. Hence the rank-$d$ component is
\[
\kappa_{\mathrm{tot}}^{-1} \circ \rho_d \circ (\id,\delta_d)_*=\Delta_d,
\]
the last equality by the definition of $\Delta_d$ in Section~\ref{sec:induced}. Summing over $d$ proves the proposition.
\end{proof}

\begin{proof}[Proof of Theorem~\ref{thm:main}]
Transport the Quillen Hopf structure to $\mc H$ through the identification \eqref{eq:identification}. By Theorem~\ref{thm:amp}, the sharbly product and coproduct agree with the Steinberg product and coproduct of Section~\ref{sec:induced}. By Propositions~\ref{prop:bcgp-product} and~\ref{prop:bcgp-coproduct}, the same is true of the Quillen product and coproduct. Thus the two structures have the same product and coproduct. Their units and counits agree by uniqueness, as do their antipodes, and the theorem follows.
\end{proof}

\begin{remark}
The same proof gives a PID version of Theorem~\ref{thm:main}. Let $R$ be a PID with fraction field $F$, and define
\[
\mc H_R:=\bigoplus_{n,k\geq 0}H_k\bigl(\GL_n(R);\St(F^n)\bigr),
\]
where $\St(F^n)$ denotes the rational Steinberg module. Under the canonical identification of $\mc H_R$ with the $E^1$-page of the rank filtration on $BK(R)$, the sharbly Hopf structure for the trivial character coincides with the Hopf structure induced by the filtered Waldhausen product and coproduct.

Indeed, the sharbly Hopf structure is defined for every PID \cite[Theorem~3.1]{ash2024hopf}, while the filtered Waldhausen construction and the description of its $E^1$-page apply to Dedekind domains \cite[Section~1.7 and Remark~2.16]{brown2024hopf}. Over a PID, every finitely generated projective module is free, so the sum over projective-module isomorphism classes in the Dedekind-domain description of the $E^1$-page has one summand in each rank. Moreover, saturated submodules of $R^n$ are direct summands and form a single $\GL_n(R)$-orbit in each rank, and saturated submodules correspond to subspaces of $F^n$. Thus the rigidified model, the induced-module decomposition, and the comparisons of the product and coproduct carry over with $R$ and $F$ in place of $\Z$ and $\Q$. The proof of Proposition~\ref{prop:sharbly} also carries over, since the stabilizer of a spanning sharbly is finite.
\end{remark}

\printbibliography

\end{document}